\documentclass{amsart}
\usepackage[sorted-cites]{amsrefs}
\usepackage{graphicx}
\usepackage{latexsym}
\usepackage{amsthm}

\usepackage{amsfonts}
\usepackage{xcolor}
\usepackage[all]{xy}
\usepackage{dsfont}

\usepackage{amssymb, mathrsfs, amsfonts, amsmath}
\usepackage{amsbsy}
\usepackage{amsfonts}
\newtheorem*{acknowledgement}{Acknowledgement}
\newtheorem{corollary}{Corollary}

\newtheorem{definition}{Definition}
\newtheorem{lemma}{Lemma}
\newtheorem{proposition}{Proposition}
\newtheorem{remark}{Remark}
\newtheorem{theorem}{Theorem}

\numberwithin{equation}{section}

\makeatletter
\@namedef{subjclassname@2020}{%
	\textup{2020} Mathematics Subject Classification}
\makeatother

\title[Noncompact quasi-Einstein manifolds]{Remarks on potential functions\\ of noncompact quasi-Einstein manifolds}
\author{Jaciane Gon\c{c}alves}

\address{Universidade Estadual do Cear\'a, Centro de Ciências e Tecnologia, Itaperi, 60714-903, Fortaleza-CE, Brazil}

\email{jaciane.goncalves@uece.br}
\thanks{J. Gon\c calves was partially supported by CNPq/Brazil [Grant: 141450/2023-1]}
\thanks{Corresponding Author: J. Gon\c calves (jaciane.goncalves@uece.br)}

\keywords{Quasi-Einstein manifolds; Einstein metrics; Asymptotically flat manifolds; Warped product metrics}
\subjclass[2020]{Primary 53C20, 53C25; Secondary 53C65.}

\date{\today}

\begin{document}	
	
	\begin{abstract}
		In this article, we study the set of potential functions on noncompact quasi-Einstein manifolds. We show that the space of all positive potential functions on a three-dimensional noncompact quasi-Einstein manifold has dimension at most two, and that equality holds if and only if the manifold is isometric to a product $B\times \Bbb{R}$, where $B$ is a $\lambda$-Einstein surface or one of the examples obtained by L. B\'erard Bergery and described in Besse’s book \cite[Section 9.118]{Besse}. Moreover, we prove that any asymptotically flat $n$-dimensional quasi-Einstein manifold with $\lambda=0$ is necessarily Ricci-flat.
	\end{abstract}

	\maketitle
	\section{Introduction} 
	
	A complete Riemannian manifold $(M^n,\,g)$, $n\geq 2$, is called an $m$-{\it quasi-Einstein manifold}, or simply, {\it quasi-Einstein manifold}, if there exists a smooth potential function $f$ on $M^n$ such that the $m$-Bakry--\'Emery Ricci tensor $Ric_{f}^{m}$ satisfies the equation
	\begin{equation}\label{eqqem}
		Ric_f^{m}=Ric+\nabla ^2f-\frac{1}{m}df\otimes df=\lambda g,
	\end{equation} for some constants $\lambda$ and $0<m< \infty,$ where $\nabla ^2f$ stands for the Hessian of $f$ and $Ric$ is the Ricci tensor of $g$; see, e.g., \cite{CaseShuWey}. We say that a quasi-Einstein manifold is \textit{trivial} if its potential function $f$ is constant, otherwise, we say that it is \textit{nontrivial}. Observe that the triviality forces $M^n$ to be an Einstein manifold. The classification of nontrivial quasi-Einstein manifolds is a central problem in understanding the existence of Einstein warped product metrics on a given manifold. For further background, see, e.g., \cite{Ernani2,BRJ,BGKW,Rondinelli,Besse,Case, CaseShuWey, catinoetal,CRZ2024, ChengRZhou,He2012, He2014, Ernaniekati, Rimoldi, Wang1,Wang,WW}.
	
	Following He, Petersen and Wylie \cite{He2012,He2014}, it is convenient to rewrite \eqref{eqqem} in terms of the function $u=e^{-\frac{f}{m}}$, which yields
	\begin{equation}\label{eqfund}
		\nabla^2u=\frac{u}{m}(Ric-\lambda g).
	\end{equation}
	This formulation is particularly suitable for studying quasi-Einstein manifolds with boundary. In this setting, a complete Riemannian manifold $(M^n,g)$ with nonempty boundary $\partial M$ is called an {\it $m$-quasi-Einstein manifold} if there exists a smooth potential function $u$ satisfying
	\begin{equation}
		\label{eq-qE}
		\left\{%
		\begin{array}{lll}
			\displaystyle \nabla^{2}u = \dfrac{u}{m}(Ric-\lambda g) & \hbox{in $M,$} \\
			\displaystyle u>0 & \hbox{on $int(M),$} \\
			\displaystyle u=0 & \hbox{on $\partial M.$} \\
		\end{array}%
		\right.
	\end{equation}
	
	A fundamental motivation for studying quasi-Einstein manifolds lies in their close relationship with Einstein warped products. When $m$ is a positive integer, $(M^n,\,g,\,u,\,\lambda)$ is an $m$-quasi-Einstein manifold (with or without boundary) if and only if the warped product $M^n\times_u F^m$ is an $\lambda$-Einstein manifold, where $F$ is a $\mu$-Einstein manifold such that
	\begin{equation}
		\label{eq:mu}
		\mu = u \Delta u + (m - 1)|\nabla u|^2 + \lambda u^2.
	\end{equation}
	For more details, see \cite[p.~267]{Besse} and \cite{He2012, He2014, KK, CaseShuWey}. In particular, (\ref{eq:mu}) can be seen as an ``integrability condition" for quasi-Einstein manifolds. Moreover, the quasi-Einstein framework naturally includes several important geometric structures. For instance, when $m=1$, one recovers \textit{static spaces} by additionally requiring $\Delta u+\lambda u=0.$ In the limiting case $m\rightarrow\infty$, equation \eqref{eqqem} corresponds precisely to gradient Ricci solitons. When $m = 2,$ a quasi-Einstein manifold corresponds to a {\it near-horizon geometry}; see \cite{BGKW,KL09}. Nontrivial examples of both compact and noncompact quasi-Einstein manifolds can be found in, e.g., \cite{BRJ, Besse, Bohm1, Bohm2, Case, CaseShuWey, catinoetal, He2012, He2014, LuePage, Ernaniekati, Rimoldi, Wang1,WW}, including the examples by L. Bergery \cite{Bergery}, see \cite[p.~271]{Besse}.

	In \cite{HPW2015}, He, Peteresen and Wylie proved the uniqueness of Einstein warped product metrics $M^n\times_uF^m$ when the base manifold $(M^n,\,g_{_M})$ is fixed. To establish this result, they studied the space of solutions of (\ref{eqfund}) that may or not change sign
	\begin{align}
		\mathcal{W}:=\mathcal{W}_{\lambda, n+m}(M,g)=\left\{u \in C^{\infty}(M): \nabla^2u=\frac{u}{m}(Ric -\lambda g)\right\}.
	\end{align} Note that $\mathcal{W}$ is clearly a vector space of functions. When $m$ is a positive integer, $(n+m)$-dimensional Einstein warped product metrics with base $(M^n, g)$ correspond to nonnegative functions in $\mathcal{W}$ satisfying $u^{-1}(0)=\partial M$. Therefore, the uniqueness of Einstein warped products, as well as of $m$-quasi-Einstein manifolds, is closely related to the dimension of $\mathcal{W}$. In the compact case, if there exists a positive function in $\mathcal{W}$, then $\dim (\mathcal{W}) = 1$ (see \cite[Corollary 1.1]{HPW2015}). Moreover, Proposition $1.7$ in \cite{HPW20152} shows that when the manifold has a nonempty boundary, whether it is compact or not, the subspace of functions in $\mathcal{W}$ satisfying the Dirichlet boundary condition $(u = 0 \text{ on } \partial M)$ is also at most one-dimensional. In contrast, for noncompact manifolds without boundary, this does not necessarily hold. Precisely, we have:

	\begin{table}[h!]\label{table2}
		\centering
		\begin{tabular}{|c|c|c|c|c|}
			\hline
			$M$ & $g$ & $u$ & $\lambda$ & $\mu$ \\ 
			\hline
			$\mathbb{R}$	& $dr^2$ & $u(r)=e^r$ & $-m$ & $0$  \\ 
			\hline
			$\mathbb{R}$	& $dr^2$  & $u(r)=\cosh(r)$ & $-m$  & $-(m-1)$ \\ 
			\hline
			$B\times \mathbb{R} $ & $g_{_B}+dr^2$& $u(x,r)=e^r$ & $-m$ & $0$\\
			\hline
			$B\times\mathbb{R}$ & $g_{_B}+dr^2$ & $u(x,r)=\cosh(r)$ & $-m$& $-(m-1)$\\
			\hline
		\end{tabular}
		\caption{Examples of $m$-quasi-Einstein manifolds admitting two linearly independent potential functions. $B$ is a $\lambda$-Einstein manifold of dimension $q \geq 2$.}
	\end{table}

	Motivated by this discussion and inspired by the work of Miao and Tam \cite{Miao-Tam}, who also investigated the question of ``how many" static potentials may exist, in the first part of this article, we investigate the number of linearly independent potential functions that a noncompact $m$-quasi-Einstein manifold without boundary can admit. The set of such functions is a subset of $\mathcal{W}$, which is denoted by
	$$\mathcal{W}^{+}=\{u\in \mathcal{W}:\,\, u>0 \}.$$
	Although $\mathcal{W}$ is a vector space, its subset $\mathcal{W}^{+}$ is not a vector subspace. Nevertheless, we refer to the number of linearly independent functions contained in $\mathcal{W}^{+}$ as the ``dimension” of $\mathcal{W}^{+}$ and we also denote it by $\dim (\mathcal{W}^{+})$. Since the classification of $1$ and  $2$ dimensional $m$-quasi-Einstein manifolds has already been fully established (see  \cite[p. 267-272]{Besse} and \cite[Appendix A]{He2012}), in our first result, we focus on the three-dimensional noncompact case. More precisely, we have established the following sharp gap theorem for $\dim (\mathcal{W}^{+}).$
	
	\begin{theorem}\label{theorem1}
		Let $(M^3,\, g,\, u,\, \lambda)$ be a nontrivial complete, connected, noncompact $3$-dimensional $m$-quasi-Einstein manifold without boundary and $m>1$. Then we have:
		\begin{equation}\label{dimensaoW^{+}}
			\dim (\mathcal{W}^+) \leq 2. 
		\end{equation}
		Moreover, if $M$ is simply connected, then $\dim(\mathcal{W}^+) = 2$ if and only if $(M^3, g, u, \lambda)$ is isometric to one of the following examples:
		\begin{itemize}
			\item[(i)]
			$\left(B^2\times \mathbb{R}, \ g=g_{_B}+c^2dr^2, \ u(r)=cv(r), \ \lambda=-\frac{m}{c^2}\right),$
		\end{itemize}
		where $B^2$ is a $\lambda$-Einstein manifold and $c$ is a positive constant.
		\begin{itemize}
			
			\item[(ii)] $(\mathbb{R}^2\times \mathbb{R},\ g=dx^2+f'(x)^2dy^2+f^2(x)dr^2,\ u(x,y,r)=f(x)v(r),\ \lambda=-(m+2))$, where
			\[
			f'^2=-1+f^2+2\frac{m^{m}}{(m+2)^{m+2}}\, f^{\,-m},\qquad 
			f(0)=1,\qquad f'>0.
			\]
			
			\item[(iii)] $(\mathbb{R}^2\times\mathbb{R},\,g=dt^2+b^2f'(t)^2d\theta^2+f^2(t)dr^2,\, u(t,\theta,r)=f(t)v(r),\, \lambda=-(m+2))$, where
			\[
			\frac{1}{b}=\frac{m+2}{2}a-\frac{m}{2a},
			\]
			$a$ is any real number with $a>\sqrt{\frac{m}{m+2}}$, and
			\[
			f'^2=-1+f^2-\left(a^{m+2}- a^{m}\right)f^{\,-m},\qquad 
			f(0)=a,\qquad f'\ge 0.
			\]
		\end{itemize}
		
		In cases \textup{(i)}-\textup{(iii)}, the function $v(r)$ is one of the solutions
		\begin{align*}
			v(r)=e^{r}\quad \text{or}\quad v(r)=\cosh(r).
		\end{align*}
		
	\end{theorem}
	
	We highlight that He, Petersen, and Wylie \cite{HPW2015} established two fundamental results concerning the space $\mathcal{W}_{\lambda,n+m}(M,g)$. They proved that if $(M^n,\,g)$ is complete, then 
	$$\dim \mathcal{W}_{\lambda,n+m}(M,\,g) \leq n+1,$$
	and equality holds if and only if $M^n$ is either a simply connected space of constant curvature or a circle. Moreover, they showed that if $(M^n,\,g)$ is simply connected and $\dim \mathcal{W}_{\lambda,n+m}(M,\,g) \geq n$, then necessarily $\dim \mathcal{W}_{\lambda,n+m}(M,\,g) = n+1.$ In particular, for simply connected manifolds, the case $\dim \mathcal{W}_{\lambda,n+m}(M,\,g) = n$ cannot occur.

	As a consequence of the proof of Theorem \ref{theorem1}, following the notation in \cite{HPW2015}, we obtain a new gap result for $\dim \mathcal{W}_{\lambda,3+m}(M^3,\,g)$. To be precise, we have the following corollary.

	\begin{corollary}\label{cor1}
		Let $(M^3,\,g)$ be a complete, connected Riemannian manifold. If there exists a positive function $ u \in \mathcal{W}_{\lambda,3+m}(M,\,g) $ and $(M^3,g)$ is not Einstein, then we have
		\begin{align}\label{eqcor}
			\dim \mathcal{W}_{\lambda,3+m}(M,\,g) \leq 2.
		\end{align}
		Moreover, if equality holds in (\ref{eqcor}), then
		\begin{align*}
			M=B^2/\Gamma\times_f F^1,
		\end{align*}
		where 
		\begin{enumerate}
			\item [(1)] $B$ is a manifold possibly with boundary, $\Gamma=\pi_{1}(M),$ and $f$ is a nonnegative function in $\mathcal{W}_{\lambda, 2+(m+1)}(B/\Gamma, g_{_{B/\Gamma}})$ with $f^{-1}(0)=\partial( B/\Gamma)$.
			\item[(2)] $f$ spans $\mathcal{W}_{\lambda, 2+(m+1)}(B/\Gamma,\, g_{_{B/\Gamma}})$, and
			\item[(3)] $F^1$ is a space form with $\dim \mathcal{W}_{\mu_{B/\Gamma}(f), 1+m}(F, g_{_F})=2$, where $\mu_{B/\Gamma}$ denotes (\ref{eq:mu}) on $\mathcal{W}_{\lambda, 2+(m+1)}(B/\Gamma, g_{_{B/\Gamma}})$. 
			\item[(4)] $\mathcal{W}_{\lambda, 3+m}(M,\,g)=\{fv: v\in \mathcal{W}_{\mu_{B/\Gamma}(f), 1+m}(F,\, g_{_F})\}$.
		\end{enumerate}
	\end{corollary}

	\begin{remark}
		Since space forms are Einstein manifolds, excluding the Einstein case automatically rules out the possibility $\dim \mathcal{W}_{\lambda,3+m}(M,\,g) = 4$. 
		A novel feature of our result is that it also excludes the case $\dim (\mathcal{W}) = 3$ without assuming that the manifold is simply connected, relying instead on the natural condition that $ \mathcal{W} $ contains a positive function.
	\end{remark}

	In the second part of this article, we use the asymptotic behavior of the potential quasi-Einstein established by Wang \cite[Theorem 5.1]{Wang} to prove a rigidity result for asymptotically flat $m $-quasi-Einstein manifolds without boundary and with $\lambda = 0$.
	As observed by He, Petersen, and Wylie \cite[p.~272]{He2012}: “\textit{from the relativity perspective, it is natural to assume in addition that the metric is asymptotically flat}.” Our approach is also inspired by the work of Miao and Tam \cite{Miao-Tam}, who used the asymptotic behavior of static potentials to classify three-dimensional asymptotically flat static manifolds. We highlight that the classification of three-dimensional asymptotically flat static manifolds with boundary was also investigated by Bunting and Masood-ul-Alam \cite{Bunting}. Related classification results on asymptotically flat static manifolds in dimension $3\leq n\leq 7$ were recently established by Harvie--Wang \cite{Harvie}. Since $m$-quasi-Einstein metrics with $m>1$ can be regarded as a natural generalization of static metrics, it is natural to seek a similar result in this more general context. 
	
	Our next result is the following.

	\begin{theorem}
		\label{theorem2}
		Let $(M^n,\, g,\, u,\, \lambda = 0)$, $n\geq 3$, be a complete (without boundary) asymptotically flat $m$-quasi-Einstein manifold with $m > 1$. Then $(M^n,\, g)$ is Ricci flat.
	\end{theorem}
	As a direct consequence of Theorem \ref{theorem2}, we get the following corollary.

	\begin{corollary}\label{cor2}
		Let $(M^3,\, g,\, u,\, \lambda = 0)$ be a three-dimensional, complete (without boundary), asymptotically flat, and simply connected $m$-quasi-Einstein manifold with $\lambda = 0$ and $m > 1$. Then $M^3$ is isometric to $\mathbb{R}^3$.
	\end{corollary}
	
	\vspace{0.20cm}
	
	The remainder of the article is organized as follows. In Section~\ref{sec2}, we recall the definition of asymptotically flat manifolds and describe the behavior of certain geometric quantities, as well as some basic facts about $m$-quasi-Einstein manifolds. Section~\ref{keyresults} collects key results that will be used in the proof of Theorem~\ref{theorem1}. In Section~\ref{proofmainresults}, we present the proofs of Theorems~\ref{theorem1} and~\ref{theorem2}, and Corollary~\ref{cor1}.

	\section{Background}
	\label{sec2}
	In this section, we review some basic facts on asymptotically flat $n$-manifold and highlight several key features of $m$-quasi-Einstein manifolds.

	We start by recalling the notion of asymptotically flat manifold. Throughout this paper, we adopt the following definition.
	
	\begin{definition}[\cite{defAF,Miao-Tam}]
		A Riemannian $n$-manifold $ (M,\, g) $, $n\geq 3$, possibly with boundary, is said to be \emph{asymptotically flat} if there exists a compact subset $ K \subset M $ such that the complement $ M \setminus K $ is a disjoint union of finitely many connected components $ E_1, \dots, E_k $, called the \emph{ends} of $M$. 
		
		Each end $E_i$ is diffeomorphic to $\mathbb{R}^n$ minus a ball, and under such a diffeomorphism, the metric $g$ satisfies
		\begin{equation}\label{expressaog}
			g_{ij} = \delta_{ij} + b_{ij}, \quad \text{with } b_{ij} = O_2(|x|^{-\tau})
		\end{equation}
		for some constant $\frac{n-2}{2}<\tau\leq n-2$, where $x = (x_1, x_2, ..., x_n)$ denotes the standard coordinates in $ \mathbb{R}^n$. 
		
		The notation  $\phi = O_l(|x|^{-\tau})$ means that for some constant $C > 0$,
		$$
		|\partial^i \phi(x)| \leq C |x|^{-\tau - i}, \quad \text{for} \quad 0\leq i\leq l.
		$$
	\end{definition}

	In order to analyze geometric quantities on an asymptotically flat manifold, it is important to recall how standard objects such as the inverse metric, the Christoffel symbols, and the Ricci tensor behave near at infinity. The following proposition collects some basic asymptotic estimates for these quantities on the ends of an asymptotically flat manifold. For brevity, we omit the details and leave the verification to the reader.

	\begin{proposition}\label{proposition1}
		If $(M^n,\,g)$ is asymptotically flat of order $\tau$. Then we have:
		\begin{itemize}
			\item [(i)] $g^{ij}=\delta_{ij}+O_2(r^{-\tau});$\\
			\item[(ii)] $\Gamma_{ij}^{k}=O_1(r^{-\tau-1});$\\
			\item[(iii)] $Ric=O(r^{-\tau-2}).$\\
		\end{itemize}
	\end{proposition}

	As stated in the main results, we focus on the case $m>1,$ since the case $m=1$ corresponds to static spaces, which were already studied by Miao--Tam \cite{Miao-Tam}, whereas the limit $m\to \infty$ corresponds to Ricci solitons, whose behavior is substantially different.

	We now recall that the fundamental equation of an $m$-quasi-Einstein manifold $(M^n,\,g,\,u,\, \lambda)$, possibly with boundary, is given by
	\begin{align}\label{eqfundamental}
		\nabla^2u=\frac{u}{m}(Ric-\lambda g),
	\end{align}
	where $u>0$ in the interior of $M^n$ and $u=0$ on $\partial M$. By tracing (\ref{eqfundamental}), we obtain
	\begin{align}\label{eqlaplaciano}
		\Delta u=\frac{u}{m}(R-n\lambda).
	\end{align}
	Combining Eqs. (\ref{eqfundamental}) and (\ref{eqlaplaciano}) we get
	\begin{align}\label{riccisemtraco}
		u\mathring{Ric}=m\mathring{\nabla}^2u,
	\end{align}
	where $\mathring{T}=T-\frac{tr T}{n}g$ stands for the traceless part of $T$.

	We now collect some well-known features for quasi-Einstein manifolds (cf. \cite{CaseShuWey,He2012, He2014,KK}).

	\begin{proposition}\label{proposition2}
		Let $(M^n,g,u, \lambda)$, $n\geq 3$, be an $n$-dimensional quasi-Einstein manifold. Then we have:
		\begin{equation}\label{ricnablau}
			\frac{1}{2}u\nabla R= -(m-1)Ric(\nabla u)-(R-(n-1)\lambda)\nabla u;
		\end{equation}
		\\
		\begin{equation}
			\label{eqmu1}
			u\Delta u+(m-1)|\nabla u|^2+\lambda u^2=\mu,
		\end{equation}
		where $\mu$ is a constant;
		\\
		\begin{align}\label{laplacianoR}
			\frac{1}{2}\Delta R+\frac{m+2}{2u}\langle \nabla u, \nabla R\rangle = & -\frac{m-1}{m}\left |Ric -\frac{R}{n}g\right |^2\nonumber \\
			&-\frac{(n+m-1)}{mn}(R-n\lambda)\left (R-\frac{n(n-1)}{m+n-1}\lambda \right).
		\end{align}
	\end{proposition}

	\section{Some Key Results}
	\label{keyresults}

	In this section, we present some key results that will play a fundamental role in the proof of Theorem \ref{theorem1}. We start with a lemma that provides some properties involving the eigenvalues of the Ricci tensor of a three-dimensional manifold and functions in $\mathcal{W}.$ Its proof is based on ideas by Tod \cite{Tod}, see also \cite{Miao-Tam}.

	\begin{lemma}\label{lemma1}
		Let $(M^3,\,g)$ be a three-dimensional connected manifold, and let $\{e_1,e_2,e_3\}$ be an orthonormal frame at a point $p \in M$, which diagonalizes the Ricci tensor.
		
		\begin{itemize}
			\item[(i)] Suppose $u \in \mathcal{W}$. Then the following identities hold:
			\begin{align*}
				u(\nabla_{e_2}R_{21}-\nabla_{e_1}R_{22}+\nabla_{e_1}R_{33}-\nabla_{e_3}R_{31}) &= (m+1)(R_{22}-R_{33})\nabla_{e_1}u, \\
				u(\nabla_{e_1}R_{12}-\nabla_{e_2}R_{11}+\nabla_{e_2}R_{33}-\nabla_{e_3}R_{32}) &= (m+1)(R_{11}-R_{33})\nabla_{e_2}u, \\
				u(\nabla_{e_1}R_{13}-\nabla_{e_3}R_{11}+\nabla_{e_3}R_{22}-\nabla_{e_2}R_{23}) &= (m+1)(R_{11}-R_{22})\nabla_{e_3}u.
			\end{align*}
			
			\item[(ii)] Suppose $\{R_{11}, R_{22}, R_{33}\}$ are distinct at the point $p \in M$, and let  $u\in \mathcal{W}^+$ and $N\in \mathcal{W}$. Then $N = ku$ in $M$ for some constant $k$.
			
			\item[(iii)] Let $u$ $\in \mathcal{W}^{+}$ and $N\in \mathcal{W}$, and define $Z=u^{-1}N$. Then, at each point in $M$ where $R_{11}=R_{22}\neq R_{33}$, we have $\nabla_{e_1}Z=\nabla_{e_2}Z=0$. 
		\end{itemize}
	\end{lemma}

	\begin{proof}
		To prove the first assertion, let $\{a, b, c, \dots\}$ denote indices in $\{1,2,3\},$ and $u \in \mathcal{W}.$  Differentiating (\ref{eqfundamental}), we obtain
		\begin{align}\label{eq2.8}
			\nabla_c \nabla_a \nabla_b u = \frac{\nabla_c u}{m}( R_{ab}-\lambda g_{ab}) + \frac{u}{m} \nabla_c R_{ab}.
		\end{align} By using the Ricci identity ($\nabla_c \nabla_b \nabla_a u - \nabla_b \nabla_c \nabla_a u = R_{cbad} \nabla_d u$), one obtains that
		\begin{align}\label{eq2.9}
			R_{cbad} \nabla_d u = \frac{\nabla_c u}{m}( R_{ab} -\lambda g_{ab})- \frac{\nabla_b u}{m} (R_{ac}-\lambda g_{ac}) + \frac{u}{m}(\nabla_c R_{ab} - \nabla_b R_{ac}).
		\end{align}
		
		At the same time, in dimension $n=3,$ the curvature tensor can be expressed in terms of the Ricci tensor as
		\begin{align}\label{eq2.10}
			R_{cbad}=R_{ca}g_{bd}+R_{bd}g_{ca}-R_{cd}g_{ba}-R_{ba}g_{cd}-\frac{R}{2}(g_{bd}g_{ca}-g_{cd}g_{ba}).
		\end{align} Substituting this into \eqref{eq2.9}, and rearranging terms, we get 
		\begin{eqnarray*}
			g_{ac} R(e_b, \nabla u) - g_{ab} R(e_c, \nabla u) &=& \frac{1}{2} R(\nabla_b u \, g_{ac}-\nabla_c u \, g_{ab}) + \frac{(m+1)}{m} \nabla_c u \, R_{ab} \nonumber\\&&- \frac{(m+1)}{m} \nabla_b u \, R_{ac}
			+ \frac{u}{m} (\nabla_c R_{ab} - \nabla_b R_{ac})\\&&+\frac{\lambda}{m}(\nabla_bug_{ac}-\nabla_cug_{ab}).
		\end{eqnarray*}
		Now, choosing $a = b \ne c$ and using the fact that $\{e_1, e_2, e_3\}$ diagonalizes Ricci, we infer
		
		\begin{align}\label{eq2.11}
			- R(e_c, \nabla u) + \frac{1}{2} R \nabla_c u = \frac{(m+1)}{m} \nabla_c u \, R_{aa} + \frac{u}{m} (\nabla_c R_{aa} - \nabla_a R_{ac}) -\frac{\lambda}{m}\nabla_cu.
		\end{align}
		
		\vspace{0.20cm}
		
		From now on, we analyse each case separately. 
		\vspace{0.20cm}
		
		\textbf{Case 1:} $a = b = 2$ and $a = b = 3$ with $c = 1$. \\
		Observe that
		\begin{align*}
			- R(e_1, \nabla u) &= - \sum_{d=1}^3 \langle \nabla u, e_d \rangle R(e_1, e_d) = - \nabla_{e_1} u \, R_{11}
		\end{align*}
		Then, Eq. \eqref{eq2.11} becomes
		\begin{align}\label{eq2.12}
			- \nabla_{e_1}u \left( R_{11} - \frac{1}{2} R + \frac{(m+1)}{m} R_{22}-\frac{\lambda}{m} \right) = \frac{u}{m} (\nabla_{e_1}R_{22}-\nabla_{e_2}R_{21}).
		\end{align}
		Similarly, for $a = b = 3$, $c = 1$, one sees that
		\begin{align}\label{eq2.13}
			- \nabla_{e_1}u \left( R_{11} - \frac{1}{2} R + \frac{(m+1)}{m} R_{33} -\frac{\lambda}{m}\right) = \frac{u}{m} (\nabla_{e_1}R_{33} - \nabla_{e_3}R_{31}).
		\end{align}
		Combining \eqref{eq2.12} and \eqref{eq2.13}, we infer
		\begin{align}
			\label{lkmjop11}
			u(\nabla_{e_2}R_{21} - \nabla_{e_1}R_{22} + \nabla_{e_1}R_{33} - \nabla_{e_3}R_{31}) = (m+1)(R_{22} - R_{33})\nabla_{e_1}u.
		\end{align}

		\vspace{0.20cm}
		
		\textbf{Case 2:} $a = b = 1$ and $a = b = 3$ with $c = 2$. \\
		From \eqref{eq2.11}, we have
		\begin{align}\label{eq2.15}
			- \nabla_{e_2}u \left( R_{22} - \frac{1}{2} R + \frac{(m+1)}{m} R_{11}-\frac{\lambda}{m} \right) = \frac{u}{m} (\nabla_{e_2}R_{11} - \nabla_{e_1}R_{12}),
		\end{align}
		\begin{align}\label{eq2.16}
			- \nabla_{e_2}u \left( R_{22} - \frac{1}{2} R + \frac{(m+1)}{m} R_{33}-\frac{\lambda}{m} \right) = \frac{u}{m} (\nabla_{e_2}R_{33} - \nabla_{e_3}R_{32}).
		\end{align}
		Combining \eqref{eq2.15} and \eqref{eq2.16}, we then obtain
		\begin{align}
			\label{plklppp11233}
			u(\nabla_{e_1}R_{12} - \nabla_{e_2}R_{11} + \nabla_{e_2}R_{33} - \nabla_{e_3}R_{32}) = (m+1)(R_{11} - R_{33}) \nabla_{e_2}u.
		\end{align}
		
		\vspace{0.20cm}
		
		\textbf{Case 3:} $a = b = 1$ and $a = b = 2$ with $c = 3$. \\
		Similarly, one obtains that
		\begin{align*}
			- \nabla_{e_3}u \left( R_{33} - \frac{1}{2} R + \frac{(m+1)}{m} R_{11}-\frac{\lambda}{m} \right) = \frac{u}{m} (\nabla_{e_3}R_{11} - \nabla_{e_1}R_{13}),
		\end{align*}
		\begin{align*}
			- \nabla_{e_3}u \left( R_{33} - \frac{1}{2} R + \frac{(m+1)}{m} R_{22} -\frac{\lambda}{m}\right) = \frac{u}{m} (\nabla_{e_3}R_{22} - \nabla_{e_2}R_{23}),
		\end{align*} which implies 
		\begin{align}
			\label{plklppp112}
			u(\nabla_{e_1}R_{13} - \nabla_{e_3}R_{11} + \nabla_{e_3}R_{22} - \nabla_{e_2}R_{23}) = (m+1)(R_{11} - R_{22}) \nabla_{e_3}u.
		\end{align} This completes the proof of the first assertion.

		\vspace{0.20cm}

		We now address the second assertion. Since the Ricci tensor has distinct eigenvalues at the point $p \in M$, by continuity there exists an open neighborhood $U\subset M$ of $p$ where $R_{11}$, $R_{22}$, and $R_{33}$ remain pairwise distinct. Therefore, the identities proved in item (i) hold on $U$. Applying item (i) to both $N$ and $u$ on $U$, one deduces that
		
		\begin{align*}
			\frac{N}{u}(m+1)(R_{22}-R_{33})\nabla_{e_1}u=(m+1)(R_{22}-R_{33})\nabla_{e_1}N,\\
			\frac{N}{u}(m+1)(R_{11}-R_{33})\nabla_{e_2}u=(m+1)(R_{11}-R_{33})\nabla_{e_2}N,\\
			\frac{N}{u}(m+1)(R_{11}-R_{22})\nabla_{e_3}u=(m+1)(R_{11}-R_{22})\nabla_{e_3}N.
		\end{align*}
	Since $u>0$ and $R_{11}, R_{22}, R_{33}$ are pairwise distinct on $U$, we conclude that
		$$-u^2\nabla \left(\frac{N}{u}\right)=0$$
		on $U$, which implies that $N = k u,\,k \in \mathbb{R},$ on $U$. From Proposition~2.4 of~\cite{He2012}, both $g$ and $u$ are analytic, and hence the scalar curvature $R$ is also analytic.  
		Since $N \in \mathcal{W}$, it satisfies the elliptic equation
		$$
		\Delta N - q N = 0, \quad \text{where} \quad q = \frac{R - n\lambda}{m}.
		$$ 
		Since $q$ is analytic, it follows from the analyticity of solutions to elliptic equations with analytic coefficients (see \cite{john}) that $N$ is analytic. Consequently, the function $N-ku$ is analytic on $M$ and vanishes on the open set $U$. Since $M$ is connected, we conclude that $N-ku\equiv0$ on $M$. Therefore, $N=ku$ on $M$.
		
		\vspace{0.10cm}	
		
		Finally, given $p \in M$ such that $R_{11} = R_{22} \neq R_{33}$, repeating the same argument as in the second item, we obtain that $\nabla_{e_1}Z = \nabla_{e_2}Z = 0$ at $p$. This finishes the proof of the lemma. 
	\end{proof}

	In the next lemma, we study the quotient $Z = u^{-1}N$,
	where  $u \in \mathcal{W^+}$ and  $N \in \mathcal{W}$. It will play a key role in the proof of Theorem~\ref{theorem1} and it follows the ideas outlined in \cite[Lemma~2.3]{Miao-Tam}.
	
	\begin{lemma}\label{lemma2} Let $(M^n, g)$ be a connected $n$-dimensional manifold and suppose that $N \in \mathcal{W}$ and $u \in \mathcal{W}^{+}.$ Define $Z=\frac{N}{u}$, then $Z$ is either constant or $\nabla Z$ never vanishes. 
	\end{lemma}

	\begin{proof}
		Let $\{x_1,\ldots,x_n\}$ be a local system of coordinates on $M ^n$. Since $\nabla^2N=\frac{N}{m}(Ric-\lambda g)$ and $\nabla^2u=\frac{u}{m}(Ric-\lambda g)$, one sees that  
		\begin{align*}
			\frac{N}{m}(R_{ij}-\lambda g_{ij})=N_{ij}=(uZ)_{ij}=&(u_iZ+uZ_i)_{j}\\
			=&u_{ij}Z+u_{i}Z_{j}+u_{j}Z_{i}+uZ_{ij}\\
			=&\frac{u}{m}Z(R_{ij}-\lambda g_{ij})+uZ_{ij}+\langle \nabla u, e_{i}\rangle\langle \nabla Z, e_{j}\rangle \\
			&+\langle \nabla u, e_{j}\rangle\langle \nabla Z, e_{i}\rangle\\
			=&\frac{N}{m}(R_{ij}-\lambda g_{ij})+uZ_{ij}+\langle \nabla u, e_{i}\rangle\langle \nabla Z, e_{j}\rangle \\
			&+\langle \nabla u, e_{j}\rangle\langle \nabla Z, e_{i}\rangle,
		\end{align*}
		so that
		\begin{align}\label{eq2.21}
			u\nabla^2Z(v,w)=-\langle \nabla u,v\rangle\langle \nabla Z, w\rangle-\langle \nabla u,w\rangle\langle \nabla Z,v\rangle,
		\end{align}
		for any tangent vectors $v$ and $w$. Now, suppose that $\nabla Z(p)=0$ at some point $p\in M$, and let $\gamma:I\to M$
		be an arbitrary geodesic with $\gamma(0)=p$. Setting $v=w=\gamma'$ in (\ref{eq2.21}), we obtain
		\begin{align*}
			u\nabla^2Z(\gamma',\gamma')=&-\langle \nabla u, \gamma'\rangle\langle \nabla Z, \gamma'\rangle -\langle \nabla u, \gamma'\rangle \langle \nabla Z, \gamma'\rangle\\
			=&-2\langle \nabla u, \gamma'\rangle\langle \nabla Z,\gamma'\rangle\\
			=&-2(u(\gamma(t)))'(Z(\gamma(t)))'.
		\end{align*} 
Let $h(t)=(Z\circ\gamma)'(t)$. Since $u\in\mathcal W^+$, we have $u>0$ on $M$ and therefore,
\begin{align*}
h'(t)=A(t)h(t), \qquad
A(t):=-2\frac{(u\circ\gamma)'(t)}{u(\gamma(t))},
\end{align*}
is well defined along $\gamma$. Clearly, this ODE has solution
		\begin{align*}
			h(t)=ce^{\int_{0}^{t}A(\tau)d\tau}, \ c \in \mathbb{R}.
		\end{align*} Hence, using the initial condition $h(0)=(Z(\gamma(t)))'_{|_{t=0}}=\langle \nabla Z(p), \gamma'(0)\rangle=0$, one obtains that $(Z(\gamma(t)))'=h(t)=0,$ for all $t \in I$. Since $\gamma$ was arbitrary, we conclude that $Z$ is constant along every geodesic issuing from $p$. In particular, $Z$ is constant in a neighborhood of $p$. 
		As mentioned in the proof of Lemma~\ref{lemma1}~(ii), the functions $u$ and $N$ are analytic; therefore, $Z$ is constant on $M^n,$ which finishes the proof of the lemma. 
	\end{proof}

	In \cite{HPW2015}, He, Petersen, and Wylie established the following characterization result, which will play a key role in the proof of Theorem \ref{theorem1} and Corollary \ref{cor1}.

	\begin{theorem}\cite[Theorem 2.2]{HPW2015}\label{thmHPW}
		Let $(M^n,\,g)$ be a complete, simply connected Riemannian manifold so that $\dim \mathcal{W}_{\lambda, n+m}(M,g)=k+1.$ Then we have:	
		\begin{align*}
			M=B^b\times_fF^k,
		\end{align*}
		where
		\begin{enumerate}
			\item [(1)] $B$ is a manifold possibly with boundary, and $f$ is a nonnegative function in $\mathcal{W}_{\lambda, b+(k+m)}(B,\, g_{_B})$ with $f^{-1}(0)=\partial B$.
			\item[(2)] $f$ spans $\mathcal{W}_{\lambda, b+(k+m)}(B,\, g_{_B})$.
			\item[(3)] $F^k$ is a space form with $\dim \mathcal{W}_{\mu_B(f), k+m}(F,\, g_{_F})=k+1$, where $\mu_B$ denotes (\ref{eq:mu}) on $\mathcal{W}_{\lambda, b+(k+m)}(B,\, g_{_B})$.
			\item[(4)] $\mathcal{W}_{\lambda, n+m}(M,\,g)=\{fv: v\in \mathcal{W}_{\mu_B(f), k+m}(F,\, g_{_F})\}$.
		\end{enumerate}
	\end{theorem}

	\begin{remark}
		\label{rem2}
		In the non-simply connected case, if $\mathcal{W}_{\lambda, n+m}(M,\,g)$ contains a positive function, it is always possible to obtain a warped product structure
		\begin{align*}
			M=(B/\Gamma)\times_fF^k,
		\end{align*}
		where $\Gamma=\pi_1(M)$. This splitting will also satisfy $(1)-(4)$ with $B$ replaced by $B/\Gamma$, see \cite[Remark 2.3]{HPW2015}.
	\end{remark}

	\section{Proof of the Main Results}
	\label{proofmainresults}
	
	This section is devoted to the proofs of Theorems \ref{theorem1} and \ref{theorem2}, and Corollary~\ref{cor1}.

	\subsection{Proof Theorem \ref{theorem1}}
	\begin{proof}
		Let $\{\lambda_1, \lambda_2, \lambda_3\}$ denote eigenvalues of $Ric$ in an open connected $U\subset M^3.$ We divide the proof into three cases.
		
		\vspace{0.10cm}
		
		\textbf{Case 1:} $\lambda_1, \lambda_2, \lambda_3$ are distinct in $U$.
		
		Since $(M^3,\, g,\, u,\, \lambda)$ is an $m$-quasi-Einstein manifold without boundary, then $u\in \mathcal{W}^{+}$. If there is another $f \in \mathcal{W}$, we can apply item (ii) of Lemma \ref{lemma1} to conclude that there exists a constant $c \in \mathbb{R}$ such that $f = cu$ on $M^3$. Hence, $\dim \mathcal{W} = 1$, and in particular, $\dim (\mathcal{W}^{+}) = 1$.
		
		\vspace{0.20cm}
		
		\textbf{Case 2:} $\lambda_1=\lambda_2=\lambda_3$ in $U$.
		
		In this case, we have that $Ric = \frac{R}{3}g$ and $R$ is constant in $U$. From (\ref{laplacianoR}), we have $R = 3\lambda$ or $R = \frac{6}{m+2}\lambda$ in $U.$ Since $g$ is analytic \cite[Proposition 2.4]{He2012}, one deduces that $R = 3\lambda$ or $R = \frac{6}{m+2}\lambda$ in $M^3.$ Substituting the value of $R$ into (\ref{laplacianoR}), we conclude that $(M^3,\, g)$ is an Einstein manifold. We may then invoke Proposition~2.4 of \cite{He2014}, which classifies Einstein $m$-quasi-Einstein quadruples $(M,g,u,\lambda)$ up to constant multiples of $u$ and $g$. Under the assumptions that $M$ is complete, connected, noncompact, and without boundary, the classification reduces to the following two examples: 
		\begin{enumerate}
			
			\item[(a)] $\left ((-\infty, +\infty)\times F^2, \ dr^2+e^{2r}g_{_F}, \ u(r)=e^{r}, \ \lambda=-(m+2)\right)$, where $F$ is Ricci flat;
			
			\item[(b)] $\left(\mathbb{H}^3, \ dr^2+\sinh^2(r)g_{\mathbb{S}^{2}}, \ u(r)=\cosh(r), \ \lambda=-(m+2)\right).$
		\end{enumerate}
		
		In particular, the positive quasi-Einstein potential is unique up to scaling. Consequently, $\dim(\mathcal W^{+})=1$.
		\vspace{0.20cm}

		\textbf{Case 3:} Two equal eigenvalues.
		
		Suppose $\dim(\mathcal{W})>2.$ In this case, since $u\in \mathcal{W}^{+}\subset \mathcal{W}$, there exist $u_2, u_3 \in \mathcal{W}$ so that $u, u_2$ and $u_3$ are linearly independent. Without loss of generality, assume $\lambda_1=\lambda_2\neq \lambda_3$ in $U$, and define $Z_1=\frac{u_2}{u}$ and $Z_2=\frac{u_3}{u}$. By item (iii) of Lemma \ref{lemma1}, we have
		
		\begin{align*}
			\nabla_{e_1}Z_{1}=\nabla_{e_2}Z_1=0 \ \text{and} \ \nabla_{e_1}Z_2=\nabla_{e_2}Z_{2}=0,
		\end{align*} in $U.$ This implies that $\nabla Z_{1}=(\partial_{3}Z_1)\partial_{3}$ and $\nabla Z_2=(\partial_3Z_2)\partial_3$, i.e., $\nabla Z_1$ and $\nabla Z_2$ are colinear. Moreover, $\nabla Z_2$ never vanishes on $U$. Otherwise, if
		$\nabla Z_2(p)=0$ for some $p\in U$, then Lemma~\ref{lemma2} would imply that
		$Z_2$ is constant on $M$, and hence $u_3=cu$ for some constant $c$,
		contradicting the linear independence of $u$ and $u_3$.
		
		Therefore, there exists $p\in U$ such that $\nabla Z_2(p)\neq0$. Since $\nabla Z_1(p)$ and $\nabla Z_2(p)$ are collinear, we can choose $\alpha\in\mathbb R$ such that $\nabla(Z_1+\alpha Z_2)(p)=0,$
		that is,
		\begin{align*}
			\nabla\left(\frac{u_2+\alpha u_3}{u}\right)=0 \ \text{at} \ p.
		\end{align*}
		Note that $w=u_2+\alpha u_3 \in \mathcal{W}$. But, by Lemma \ref{lemma2}, $\frac{w}{u}$ is either constant or $\nabla (\frac{w}{u})$ never vanishes. Since there exists $p \in M$ such that $\nabla(\frac{w}{u})(p)=0$, we conclude that $\frac{w}{u}$ is constant in $U$. Therefore, $u_2+\alpha u_3=\beta u$, for some constant $\beta$, which leads a contradiction. Thus, $\dim(\mathcal{W})\leq 2$ and in particular, $\dim(\mathcal{W}^{+})\leq 2$.  This completes the proof of (\ref{dimensaoW^{+}}).
		
		\vspace{0.15cm}
		
		Proceeding, we now assume that $M$ is simply connected and analyze the equality case. If $\dim(\mathcal{W}^{+}) = 2$, then Cases 1 and 2 do not occur. Hence, without loss of generality, we may assume that $\lambda_1 = \lambda_2 \neq \lambda_3$ on $U$. If $\dim(\mathcal{W}) > 2$, we can repeat the argument from Case 3 to obtain a contradiction. Therefore, $\dim(\mathcal{W}) \leq 2$. Since $\mathcal{W}^{+}$ is a subset of $\mathcal{W}$, we conclude that $\dim(\mathcal{W}) = 2$. By Theorem \ref{thmHPW}, $(M^3,\,g)$ splits as 
		\begin{align*}
			M = B^2 \times_f F^1,
		\end{align*}
		where $f$ satisfies properties (1)–(4) of Theorem \ref{thmHPW}.

		Now we determine the structures of $B^2$ and $F^1$ separately. As shown by Qian \cite{Qian} and Kim--Kim \cite{KK}, an $m$-quasi-Einstein manifold is compact if and only if $\lambda>0$. Consequently, since $M$ is noncompact, we must have $\lambda \le 0$. From Assertion (1) in Theorem \ref{thmHPW}, it follows that $B^2$ is also noncompact with $\lambda\leq 0$ and has no boundary. Next, if $f$ is trivial, then (1) in Theorem \ref{thmHPW} immediately implies that $B^2$ is $\lambda$-Einstein. Otherwise, if $f$ is nontrivial, we invoke the classification of noncompact without boundary, $2$-dimensional quasi-Einstein manifolds given in \cite[ 9.118(a)-(d)]{Besse} (see also \cite{Bergery} and \cite[Appendix A]{He2012}) to conclude that $(B^2,\,g_{_B}, f, \lambda)$ is isometric to precisely one of the examples listed below:
		\begin{itemize}
			\item[(a)] $(\mathbb{R}^2, \ dt^2+\frac{4f'(t)^2}{(p-1)^2}d\theta^2, \ f=f(t), \ \lambda=0,\,\mu(f)=p-1)$, where
			\begin{align*}
				f(0)=1,\quad f'\geq 0,\quad (f')^2=1-f^{1-p}.
			\end{align*}
			
			\item[(b)] $(\mathbb{R}^2, \, dx^2+e^{2x}dy^2, \, f(x)=e^x, \, \lambda=-(p+1),\,\mu(f)=0);$
			
			\item[(c)] $(\mathbb{R}^2, \, dx^2+f'(x)^2dy^2, \ f=f(x), \, \lambda=-(p+1),\,\mu(f)=1-p)$, where
			\begin{align*}
				f'^2=-1+f^2+2\frac{(p-1)^{p-1}}{(p+1)^{p+1}}\, f^{1-p},\quad f(0)=1,\quad f'>0;
			\end{align*}
			
			\item[(d)] $(\mathbb{R}^2, \, dt^2+b^2f'(t)^2d\theta^2, \, f=f(t), \,\lambda=-(p+1),\mu(f)=1-p)$, where
			\[
			\frac{1}{b}=\frac{p+1}{2}a+\frac{\mu(f)}{2a},
			\]
			$a$ is any real number with $a>0$ if $\mu(f)\geq 0$ and $a>\sqrt{\frac{-\mu(f)}{p+1}}$ if $\mu(f)<0$, and
			\begin{align*}
				f'^2=\frac{\mu(f)}{p-1}+f^2-\left(a^{p+1}+\frac{\mu(f)}{p-1}\, a^{p-1}\right)f^{1-p},\quad f(0)=a,\quad f'\geq 0.
			\end{align*}
			Here, $p=m+1$ and the parameters $(t,\theta)$ and $(x,y)$ represent polar and Cartesian coordinates on $\mathbb{R}^2$, respectively. In particular, as observed in \cite[Sec. 9.118]{Besse}, in the special case $\mu(f)=1-p$ and $a=1,$ we have $(\mathbb{R}^2,\,g=dt^2+ \sinh (t)^2 d\theta^2,\, \ f(t)=\cosh (t), \, \lambda=-(p+1)).$
		\end{itemize}

		On the other hand, given $v>0$ $\in$ $\mathcal{W}_{\mu_{B}(f),\, 1+m}(F,\, g_{_F})$, Proposition 3.1 (and Table 1) of \cite{He2014} guarantees that $(F^1,\, g_{_F},\, v,\, \mu_{B}(f))$ is isometric, up to scaling of $v$ and $g_{_F}$, to one of the following examples:
		
		\begin{itemize}
			\item[(e)] $(\mathbb{R}, \ dr^2,\ v(r)=e^{r},\ \mu_{B}(f)=-m);$
			
			\item[(f)] $(\mathbb{R}, \ dr^2,\ v(r)=\cosh(r),\ \mu_{B}(f)=-m).$
			
			\item[(g)] $(\mathbb{R}, \ dr^2,\ v(r)=C,\ \mu_{B}(f)=0).$
		\end{itemize}

		Proceeding, since $(M^3,\, g,\, u,\, \lambda)$ is a nontrivial $m$-quasi-Einstein manifold, property (4) of Theorem~\ref{thmHPW} implies that it is isometric to one of the following models:
		\begin{enumerate}
			\item [(A)] When $f=c$ is a positive constant:
			
			\begin{itemize}
				\item[(i)] $\left(B^2\times \mathbb{R}, \  g=g_{_B}+c^2dr^2,\  u(r)=ce^{r},\  \lambda=-\frac{m}{c^2}\right);$
				
				\item[(ii)] $\left(B^2\times \mathbb{R},\  g=g_{_B}+c^2dr^2,\  u(r)=c\cosh(r),\  \lambda=-\frac{m}{c^2}\right),$
			\end{itemize}
		\end{enumerate}
		where $B^2$ is a $\lambda$-Einstein surface and we used that $\mu_{B}(c)=\lambda c^2$.
		
		\begin{enumerate}
			\item [(B)] We now assume $f$ is nontrivial. In this case, by using (2) in Theorem \ref{thmHPW}, one sees that the fact that $\dim \mathcal{W}^{+}(M)=2$ forces $\mu_{B}(f)= -m <0$ and hence, we have the following possibilities: 
			\begin{itemize}
				\item[(iii)]    $(\mathbb{R}^2\times \mathbb{R},\ dx^2+f'(x)^2dy^2+f^2(x)dr^2,\ u(x,y,r)=f(x)v(r),\ \lambda=-(m+2))$, where
				\[
				f'^2=-1+f^2+2\frac{m^{m}}{(m+2)^{m+2}}\, f^{\,-m},\qquad 
				f(0)=1,\qquad f'>0.
				\]
				
				\item[(iv)]   $(\mathbb{R}^2\times\mathbb{R},\ dt^2+b^2f'(t)^2d\theta^2+f^2(t)dr^2,\ u(t,\theta,r)=f(t)v(r),\ \lambda=-(m+2))$, where
				\[
				\frac{1}{b}=\frac{m+2}{2}a-\frac{m}{2a},
				\]
				$a$ is any real number with $a>\sqrt{\frac{m}{m+2}}$, and
				\[
				f'^2=-1+f^2-\left(a^{m+2}- a^{m}\right)f^{\,-m},\qquad 
				f(0)=a,\qquad f'\ge 0.
				\]
			\end{itemize}
		\end{enumerate}
		In cases (iii) and (iv), the function $v(r)$ is one of the solutions
		\begin{align*}
			v(r)=e^{r}\quad \text{or}\quad v(r)=\cosh(r).
		\end{align*}
		
		Finally, it follows from (2) of Theorem \ref{thmHPW} that $\dim \mathcal{W}_{\lambda, 2+(m+1)}(B, g_{_B})=1$. Thus the converse statement is straightforward. This completes the proof of Theorem \ref{theorem1}.
	\end{proof}

	\subsection{Proof of Corollary \ref{cor1}} 
	\begin{proof}
		If $M^n$ is compact, the existence of a positive function in $\mathcal{W}_{\lambda, 3+m}(M,\,g)$ implies, by Corollary~1.1 in~\cite{HPW2015}, that $\dim( \mathcal{W}) =1$. Hence, we may assume that $M^n$ is noncompact.  
		
		Let $\lambda_1, \lambda_2, \lambda_3$ be the eigenvalues of the Ricci tensor on an open connected set $U \subset M$. Since $ (M^n,\,g)$ is not Einstein, the case $ \lambda_1 = \lambda_2 = \lambda_3$ cannot occur, as shown in the proof of Theorem \ref{theorem1}. For the remaining cases, we have 
		\begin{align*}
			\dim (\mathcal{W}) \leq 2.
		\end{align*}
		
		The case of equality follows directly from Theorem \ref{thmHPW} and Remark \ref{rem2}.
	\end{proof}

	\subsection{Proof of Theorem \ref{theorem2}} 
	\begin{proof}
		
		We shall adapt the arguments of Miao-Tam \cite{Miao-Tam}. Assuming that $u$ is nontrivial, it follows from \cite{Case} that $\mu > 0.$ In this case, we may use \cite[Theorem 5.1]{Wang} to deduce that the potential function $u$ of a complete noncompact (without boundary) $m$-quasi-Einstein manifold with $\lambda=0$ and $m>1$ must satisfy 
		\begin{equation}\label{eq:pointwise_u}
			c_0 r^{\frac{m-1}{m(m+n-1)}} \leq \sup_{x\in \partial B_{p}(r)} u(x) \leq C r,
		\end{equation}
		for $|x| = r \gg 1$ (sufficiently large). This implies that $u=O(r).$  Moreover, since the manifold is connected at infinity (see \cite[Theorem 1]{Rondinelli}) and asymptotically flat, there exists a compact set $K \subset M$ such that $M \setminus K$ is diffeomorphic to $\mathbb{R}^n \setminus B_\rho$, for some $\rho>0$.  Using coordinates $x = (x_1, \dots, x_n)$ with $r = |x| > \rho$, the metric can be written as 
		\begin{align*}
			g_{ij} = \delta_{ij} + b_{ij}, \quad \text{with\,\,\ } b_{ij}= O_2(r^{-\tau}).
		\end{align*}
		Consequently, by Proposition \ref{proposition1}, the Ricci tensor satisfies
		\begin{align*}
			Ric = O(r^{-\tau-2}).
		\end{align*}  By using (\ref{eqfundamental}) and the fact that $ u=O(r) $, we infer
		\begin{align}\label{ordemhessu}
			\nabla^2 u =O(r) O\left(r^{-\tau - 2} \right) = O(r^{-\tau - 1}).
		\end{align} 
		
		On the other hand, by combining (\ref{eq:mu}) and (\ref{eqlaplaciano}), one obtains that
		
		\begin{align}\label{eqmu}
			\mu = \frac{u^2 R}{m} + (m-1)|\nabla u|^2.
		\end{align} Moreover, it is known that any $m$-quasi-Einstein manifold with $\lambda \leq 0$ must satisfy $ R \geq \lambda n$ (see \cite{Wang}), and hence, for $\lambda = 0$, we have $R \geq 0$. So, (\ref{eqmu}) yields
		\begin{align*}
			|\nabla u| \leq \sqrt{\frac{\mu}{m-1}}.
		\end{align*}
		By the asymptotic flatness condition of $g$ and the uniform equivalence of the metric to the Euclidean metric at infinity, it is then straightforward to verify that the partial derivatives $\frac{\partial u}{\partial x_i}$ are bounded, i.e, $\frac{\partial u}{\partial x_i} =O(1).$
		
		We now need to estimate the decay of the Hessian of $u$ in Euclidean coordinates. Recalling the coordinate expression for the Hessian with respect to the Euclidean connection, we write
		\begin{align*}
			\frac{\partial^2 u}{\partial x^i \partial x^j} = \nabla^2 u(e_i, e_j) - \sum_k \Gamma_{ij}^k \frac{\partial u}{\partial x^k},
		\end{align*}
		where $\Gamma_{ij}^k= O_1(r^{-\tau - 1})$ , $ \frac{\partial u}{\partial x^k}= O(1)$. Therefore, the term involving the Christoffel symbols satisfies
		\begin{align*}
			\Gamma_{ij}^k \frac{\partial u}{\partial x^k} =O(r^{-\tau - 1}).
		\end{align*}
		By (\ref{ordemhessu}), we have $\nabla^2u=O(r^{-\tau-1})$ and hence, the second derivatives of $u$ in Euclidean coordinates also satisfy
		\begin{align}\label{hessunoRn}
			\frac{\partial^2 u}{\partial x^i \partial x^j}=O(r^{-\tau - 1}).
		\end{align}
		We now claim that for each $j = 1, 2,..., n$, the limit $\lim_{x \rightarrow \infty} \frac{\partial u}{\partial x_j}$ exists and is finite. In fact, we fix a point $x_0 \in M$ with $|x_0|>\rho,$ and consider the curve 
		\begin{align*}
			\gamma(s) = x_0 + s e_i, \quad s \geq s_0 > 0,
		\end{align*}
		which lies entirely in the asymptotically flat end of $M$. Let 
		\begin{align*}
			\Phi(s) = \frac{\partial u}{\partial x_j}(\gamma(s)) = \frac{\partial u}{\partial x_j}(x_0 + s e_i).
		\end{align*}
		By the chain rule and by (\ref{hessunoRn}), we have
		\begin{align*}
			\Phi'(s) 
			&= \sum_{k} \frac{\partial^2 u}{\partial x_k \partial x_j}(x_0 + s e_i) \frac{d(x_0^k + s \delta_i^k)}{ds} \\
			&= \frac{\partial^2 u}{\partial x_i \partial x_j}(x_0 + s e_i) \\
			&= O(|x_0 + s e_i|^{-\tau - 1}) \\
			&= O(s^{-\tau - 1}).
		\end{align*} From this, it follows that
		\begin{align*}
			|\Phi(t) - \Phi(s_0)| = \left| \int_{s_0}^{t} \frac{\partial^2 u}{\partial x_i \partial x_j}(x_0 + s e_i) \, ds \right| \leq C \int_{s_0}^{t} \frac{1}{s^{\tau + 1}} \, ds = -\frac{C}{\tau} t^{-\tau} + \frac{C}{\tau} s_0^{-\tau}.
		\end{align*}
		Therefore, $\lim_{t \rightarrow \infty} \Phi(t) = \lim_{t \rightarrow \infty} \frac{\partial u}{\partial x_j}(x_0 + t e_i)$ exists and is finite. 
		We denote $\lim_{x \rightarrow \infty} \frac{\partial u}{\partial x_j}(x) = a_j$. Define $\Lambda:=\sum_{j=1}^{n}a_jx_j$, we want to study the asymptotic behavior of $h(x) := u(x) - \Lambda(x)$. Note that $\frac{\partial h}{\partial x_j}(x) = \frac{\partial u}{\partial x_j}(x) - a_j,$ which implies
		\begin{align}\label{ordemhessiana h}
			\frac{\partial^2 h}{\partial x_i \partial x_j}(x) =  \frac{\partial^2 u}{\partial x_i \partial x_j}(x) = O(|x|^{-\tau - 1}).
		\end{align}

		At the same time, $\lim_{x \rightarrow \infty} \frac{\partial h}{\partial x_j}(x) = \lim_{x \rightarrow \infty} \left( \frac{\partial u}{\partial x_j}(x) - a_j \right) = 0$. With this data,  one deduces that
		\begin{align*}
			\left| \frac{\partial h}{\partial x_j}(x_0 + s e_i) \right|
			&= \left| \lim_{b \rightarrow \infty} \int_{s}^{b} \frac{\partial^2 h}{\partial x_i \partial x_j}(x_0 + t e_i) \, dt \right| \\
			&\leq \lim_{b \rightarrow \infty} \int_{s}^{b} \left| \frac{\partial^2 h}{\partial x_i \partial x_j}(x_0 + t e_i) \right| \, dt \\
			&\leq A \lim_{b \rightarrow \infty} \int_{s}^{b} \frac{1}{t^{\tau + 1}} \, dt = \frac{A}{\tau s^{\tau}},
		\end{align*}
		where $A$ is a positive constant. Consequently,
		\begin{align}\label{ordem derivada de h}
			\frac{\partial h}{\partial x_j}(x_0 + s e_i)=O(s^{-\tau}).
		\end{align}
		As before, we fix a point $x_0 \in M$ with $|x_0|>\rho$. Consider the curve $\gamma(s) = x_0 + s e_i$ and define $\varphi(s) = h(\gamma(s)) = h(x_0 + s e_i)$. Then, $$\varphi'(s) = \frac{\partial h}{\partial x_i}(x_0 + s e_i) = O(s^{-\tau}).$$ 
		Therefore,
		\begin{align}\label{estcases}
			|\varphi(t) - \varphi(s_0)| &= \left| \int_{s_0}^{t} \varphi'(s) \, ds \right| \leq \int_{s_0}^{t} |\varphi'(s)| \, ds \leq \overline{C} \int_{s_0}^{t} \frac{1}{s^{\tau}} \, ds,
		\end{align}
		where $\overline{C}$ is a positive constant. Now we distinguish the cases according to the dimension $n$ of $M$.
		
		\begin{itemize}
			\item[(i)] Case $n\geq 4$.
			
			In this case, by the asymptotic flatness assumption,
			\begin{equation*}
			\tau>\frac{n-2}{2}\geq 1.
			\end{equation*}
			Hence $\tau>1$. From \eqref{estcases}, we obtain
			\begin{align*}
			\int_{s_0}^{t}s^{-\tau}ds=\frac{1}{1-\tau}(t^{1-\tau}-s_0^{1-\tau}),
			\end{align*}
			so that $h(x_0+te_i)-h(x_0+s_0e_i)=O(t^{1-\tau})$. Since $s_0>0$ is fixed and does not depend on $t$, the term $h(x_0+s_0e_i)$ does not affect the asymptotic behavior as $t\rightarrow \infty$. Hence,
			\begin{align*}
				u(x)-\Lambda(x)=h(x)=O(1),
			\end{align*}
			and therefore
			\begin{equation}\label{expucasei}
				u(x)=a_1x_1+\cdots+a_nx_n+O(1).
			\end{equation}
			For each unit vector $v$, consider the ray $x=tv$, with $t>\rho$. Dividing (\ref{expucasei}) by $t$ and letting $t\to\infty$, we get
			\begin{equation*}
			\lim_{t\to\infty}\frac{u(tv)}{t}=\langle a,v\rangle.
			\end{equation*}
			Since $u=e^{-f/m}>0$ on $M$, it follows that $\langle a,v\rangle\geq 0$. Replacing $v$ by $-v$, we obtain $\langle a,v\rangle\leq 0$. Thus $\langle a,v\rangle=0$ for every unit vector $v$, and hence $a=0$. Consequently, $u(x)=O(1)$, which contradicts \eqref{eq:pointwise_u}.
			
			\item[(ii)] Case $n=3$.
			
			In this case, the asymptotic flatness condition gives
			\begin{align*}
			\frac{1}{2}<\tau\leq 1.
			\end{align*}
			Assume first $\tau=1$. From (\ref{estcases}), we obtain
			\begin{align*}
			u(x)=a_1x_1+a_2x_2+a_3x_3+O(\ln |x|).
			\end{align*}
			Arguing as in case (i), we obtain $a=0$. Therefore, $u(x)=O(\ln |x|)$. In this situation, inequality (\ref{eq:pointwise_u}) implies that
			\begin{align*}
			c_0 r^{\frac{m-1}{m(m+2)}}\leq \sup_{\partial B_p(r)}u\leq C\ln (r),
			\end{align*}
			where $r=|x|>\rho$. However, this leads to a contradiction, because for any $\alpha>0$,
			\begin{align*}
				\lim_{r\rightarrow \infty} \frac{r^{\alpha}}{\ln r}=\infty,
			\end{align*}
			that is, $r^{\alpha}$ grows faster than $\ln r$. In particular, since $\frac{m-1}{m(m+2)}>0$, this contradicts the upper bound.
			
			It remains to consider the case $\frac{1}{2}<\tau<1$. From \eqref{estcases}, we have
			\begin{align*}
			u(x)=a_1x_1+a_2x_2+a_3x_3+O(|x|^{1-\tau}).
			\end{align*}
			Again, the positivity of $u$ implies, by considering the rays $tv$ and $-tv$, $|v|=1$, that $a=0$. Hence
			\begin{align*}
			u(x)=O(|x|^{1-\tau}).
			\end{align*}
		So we can estimate the Hessian of $u$ using the quasi-Einstein equation $\nabla^2u=\frac{u}{m}Ric$. Specifically,
		
		\begin{align*}
			|\nabla^2 u| &= |u Ric| \leq C r^{1 - \tau} \cdot r^{-2 - \tau} \leq C r^{-1 - 2\tau},
		\end{align*}
		so that $\nabla^2 u = O(r^{-1 - 2\tau})$.
		
		From equation (\ref{ordem derivada de h}), we already know that $\frac{\partial u}{\partial x_i}= O(r^{-\tau})$, and hence
		\begin{align*}
			\frac{\partial^2 u}{\partial x_i \partial x_j} 
			&= \nabla^2 u(e_i, e_j) + \frac{\partial u}{\partial x_k} \, \Gamma_{ij}^k  \\
			&= O(r^{-1 - 2\tau}) + O(r^{-\tau}) \cdot O(r^{-\tau - 1}) \\
			&= O(r^{-1 - 2\tau}) + O(r^{-2\tau - 1}) \\
			&= O(r^{-1 - 2\tau}).
		\end{align*}
		
		Now, integrating this decay as done previously, we obtain that
		$$
		\frac{\partial u}{\partial x_i}=O(r^{-2\tau}),
		$$
		for $i=1,2,3$, improving the estimate for the first derivatives of $u$ in this three-dimensional case. Let $v=(v_1,v_2,v_3)$ be a unit vector and $\Psi_v(s)=u(sv)$, then
		\begin{align*}
		\Psi_v'(s)
		=
		\sum_{j=1}^3 v_j\frac{\partial u}{\partial x_j}(sv)
		=
		O(s^{-2\tau}).
		\end{align*}
		Therefore,
		\begin{align*}
			|\Psi_v(t)-\Psi_v(s_0)|\leq&
			\int_{s_0}^{t}|\Psi_v'(s)|\,ds \\
			\leq&
			C\int_{s_0}^{t}s^{-2\tau}\,ds \\
			=&
			\frac{C}{1-2\tau}
			\left(s_0^{1-2\tau}-t^{1-2\tau}\right),
		\end{align*}
		which is bounded as $t\to\infty$ because, in the present case,
		$\tau>\frac{1}{2}$, and hence $2\tau>1$. Therefore, for every unit vector $v$, the limit
		$$
		\lim_{t\to\infty}u(tv),
		$$
		exists and is finite. Since $u$ is continuous and the limit exists along every ray, the supremum over $\partial B_r$ must converge to the maximum of these limits, which is finite. This contradicts (\ref{eq:pointwise_u}).
		\end{itemize}
			In conclusion, $u$ must be trivial and hence, by \eqref{eqfundamental}, one concludes that $(M^n,\,g)$, $n\geq 3$, is Ricci flat. This finishes the proof of the theorem.
	\end{proof}

	\begin{acknowledgement}
	The author would like to thank the referee for the careful reading,
relevant remarks and valuable suggestions. Moreover, the author is grateful to E. Ribeiro Jr for helpful conversations about this subject and for his constant encouragement.	\end{acknowledgement}
	

	


\end{document}